\documentclass[11pt, leqno]{amsart}
\usepackage{amsmath}
\usepackage{amsthm}
\usepackage{newtxtext, newtxmath}
\usepackage{mathrsfs}

\usepackage[colorinlistoftodos]{todonotes}
\renewcommand{\nu}{\upsilon}

\usepackage{graphicx}
\usepackage[font=small,labelfont=bf]{caption}
\usepackage{epstopdf}
\usepackage{xcolor}
\usepackage{float}
\usepackage{pgfplots}
\usepackage{listings}
\usepackage{longtable}
\usepackage{comment}
\usepackage{esint}
\usepackage{nicefrac}
\usepackage{tikz}
\usetikzlibrary{arrows, patterns, patterns.meta, calc, math, decorations.markings}
\usepackage[normalem]{ulem} 
\usepackage[parfill]{parskip}
\usepackage{esint}
\usepackage{linegoal}

\usepackage{mathtools}
\usepackage[utf8]{inputenc}
\usepackage[T1]{fontenc}
\usepackage[shortlabels, inline]{enumitem}
\DeclarePairedDelimiterX\set[1]\lbrace\rbrace{#1}

\usepackage{scalerel}[2014/03/10]
\usepackage[usestackEOL]{stackengine}

\usepackage{hyperref}
\hypersetup{
plainpages=false,
colorlinks,
linkcolor={cyan!90!black},
citecolor={magenta},
urlcolor={red!90!black},
bookmarksdepth=2,} 
\usepackage[noadjust]{cite}

\hypersetup{
pdftitle={Boundedness of the parabolic Riesz Transform},
pdfsubject={Mathematics, PDE, Analysis},
pdfauthor={Khalid Baadi, Moritz Egert, Benjamin W. Kosmala},
pdfkeywords={}
}

\newtheorem{thm}{Theorem}[section]

\newtheorem{prop}[thm]{Proposition}
\newtheorem{lem}[thm]{Lemma}

\theoremstyle{definition}
\newtheorem{defn}[thm]{Definition}

\theoremstyle{remark}
\newtheorem{rem}[thm]{Remark}

\renewcommand{\L}{\mathrm{L}}

\newcommand{\Cont}{\mathrm{C}}
\newcommand{\loc}{\mathrm{loc}}
\newcommand{\e}{\mathrm{e}}
\newcommand{\ii}{\mathrm{i}}
\renewcommand{\d}{\mathrm{d}}
\newcommand{\dd}{\,\mathrm{d}}
\renewcommand{\Im}{\mathrm{Im}}
\newcommand{\R}{\mathbb{R}}
\newcommand{\C}{\mathbb{C}}

\newcommand{\norm}[1]{\Vert#1\Vert}
\newcommand{\abs}[1]{\vert#1\vert}

\newcommand{\eps}{\varepsilon}
\renewcommand{\Re}{\mathrm{Re}}

\colorlet{shadecolor}{gray!20}
\pgfplotsset{compat=1.9}
\usepgflibrary{fpu}
\makeatletter
\let\c@equation\c@thm
\makeatother
\numberwithin{equation}{section}

\author{Pascal Auscher}
\address{Universit{\'e} Paris-Saclay, CNRS, Laboratoire de Math\'{e}matiques d'Orsay, 91405 Orsay, France}
\email{pascal.auscher@universite-paris-saclay.fr}
\author{Moritz Egert}
\address{TU Darmstadt, Fachbereich Mathematik, Schlossgartenstr.\ 7, 64289 Darmstadt, Germany}
\email{egert@mathematik.tu-darmstadt.de}
\author{Benjamin W. Kosmala}
\address{TU Darmstadt, Fachbereich Mathematik, Schlossgartenstr.\ 7, 64289 Darmstadt, Germany}
\email{kosmala@mathematik.tu-darmstadt.de}
\date{\today}
\title{Singular solutions to parabolic equations in one space dimension}

\begin{document}
\begin{abstract}
We construct singular weak solutions to scalar, linear, uniformly parabolic equations with complex coefficients in one space dimension. Our examples show that the space-time integrability provided by the energy estimates is sharp: for every $p>6$, there exists such an equation admitting an energy solution that fails to be $p$-integrable on a compact subset of the space-time domain. In particular, the local boundedness of weak solutions from the De Giorgi--Nash--Moser theory fails for complex coefficients already in one space dimension.
\end{abstract}

\maketitle

\section{Introduction}
\label{sec: intro}

In 1958, John Nash~\cite{Nash} proved that solutions to scalar, linear, uniformly parabolic equations in divergence form with real coefficients are continuous. For the corresponding class of equations with complex coefficients, Frehse and Meinel~\cite{Frehse} constructed unbounded solutions in spatial dimensions $n\geq 3$. More recently, Mooney obtained, by a different construction, unbounded solutions in dimensions $n\geq 2$ and showed, in addition, that the space-time integrability furnished by the energy estimates is sharp~\cite{Mooney}. We complete the picture by showing that Nash's continuity theorem fails for complex coefficients even in one space dimension. The one-dimensional case has remained of interest: related questions were raised explicitly in~\cite[p.~207]{Bechtel}, while other pathological phenomena were exhibited by Krylov~\cite{Krylov}.

In one space dimension the class of equations we are studying is
\begin{align}
\label{eq: PDE}
 \partial_t u - \partial_x \bigl(a(t,x) \partial_x u\bigr) = 0,
\end{align}
where $a$ is a bounded, complex, measurable and strictly accretive function. Strict accretivity means that there is a constant $\lambda>0$ such that $\Re(a) \geq \lambda$ almost everywhere. By a solution $u$ in a backward-in-time parabolic cylinder 
\begin{align}
\label{eq: cylinder}
    Q_r(t_0,x_0) \coloneqq (t_0 - r^2, t_0) \times (x_0 -r, x_0 + r)
\end{align}
we mean a function $u \in \L^2(Q_r(t_0,x_0))$ such that $\partial_x u \in \L^2(Q_r(t_0,x_0))$ and \eqref{eq: PDE} holds in $Q_r(t_0,x_0)$ in the sense of distributions. By a solution in $(-\infty, 0) \times \R$ (resp. $\R^2$) we mean a function that is a solution in all cylinders $Q_r(t_0,x_0)$ contained in that domain. For real-valued $a$, the De Giorgi--Nash--Moser bound for solutions in $Q_r(t_0,x_0)$ states that
\begin{align*}
    \|u\|_{\L^\infty(Q_{r/2}(t_0,x_0))} \leq C r^{-\frac{3}{2}} \|u\|_{\L^2(Q_{r}(t_0,x_0))},
\end{align*}
where $C$ depends only on the ellipticity parameters, see, e.g., \cite{Moser}. As our first main result we show that this bound can fail when $a$ is complex-valued.
\begin{thm}
    \label{thm: DGNM fails}
    There is a bounded, measurable and strictly accretive function $a: (-\infty,0) \times \R \to \C$ such that \eqref{eq: PDE} has a solution $u$ in $(-\infty,0) \times \R $ satisfying
    \begin{align*}
        \|u\|_{\L^\infty(Q_1(0,0))} = \infty.
    \end{align*}
\end{thm}
The energy estimates together with local higher integrability theory (e.g.~\cite[Rem.~4.5]{ABES}) guarantee that solutions to~\eqref{eq: PDE} are locally $p$-integrable for some $p>6$, depending only on the ellipticity parameters. Our second main result strengthens Theorem~\ref{thm: DGNM fails} by showing that this exponent can be arbitrarily close to $6$.
\begin{thm}
    \label{thm: Lp fails}
Let $p>6$. There is a bounded, measurable and strictly accretive function $a: \R^2 \to \C$ such that \eqref{eq: PDE} has a solution in $\R^2$ satisfying for every $\eps \in (0,1)$ that
    \begin{align*}
        \|u\|_{\L^p(Q_\eps(0,0))} = \infty.
    \end{align*}
\end{thm}

Let us explain the idea behind the proof of Theorem~\ref{thm: DGNM fails}. The same construction turns out to be flexible enough to yield Theorem~\ref{thm: Lp fails} with very little additional effort; see Section~\ref{sec: profile}.

Similar to \cite{Mooney}, our solution $u$ and coefficient $a$ in Theorem~\ref{thm: DGNM fails} are self-similar and take the form
\begin{align}
\begin{split}
\label{eq: ansatz}
    u(t,x) &= (-t)^{-m} F(\nicefrac{x}{\sqrt{-t}}),\\
    a(t,x) &= A(\nicefrac{x}{\sqrt{-t}}),
\end{split}
\end{align}
with $F$ and $A$ even.
This ansatz reduces the parabolic equation \eqref{eq: PDE} to an ordinary differential equation for $F$ on $[0,\infty)$, subject to the Neumann boundary condition $(AF')(0) = 0$; see Section~\ref{sec: profile}. To produce a singularity of $u$ at the space-time origin, we seek, for some $m>0$, a profile satisfying $|F(y)| \approx |y|^{-2m}$ for $|y| \geq 1$. With this choice, $u$ develops a singularity of the form $|u(t,x)| \approx |x|^{-2m}$ in the region above the space-time parabola $|x| = \sqrt{-t}$, which yields Theorem~\ref{thm: DGNM fails}. See Figure~\ref{fig: parabolas} for an illustration. 
\begin{figure}[htb]
    \definecolor{gold}{HTML}{F1A226}
    \colorlet{color1}{violet}
    \colorlet{color2}{gold}
    \begin{tikzpicture}[scale=0.8, every node/.style={font=\small}]

            \draw[black, thick, ->] (-10,0) -- (2,0) node[right] {$t$};
            \draw[black, thick, ->] (0,-4) -- (0,4) node[above] {$x$};

            \fill[black] (-9,0) circle (2pt);
            \node[below] at (-9.4,0) {$-1$};
            \fill[black] (-4,0) circle (2pt);
            \node[below] at (-4.5,0) {$-\eps^2$};
            \fill[black] (0,2) circle (2pt);
            \node[right] at (0,2) {$\eps$};
            \fill[black] (0,3) circle (2pt);
            \node[right] at (0,3) {$1$};
            \fill[black] (0,-2) circle (2pt);
            \node[right] at (0,-2) {$-\eps$};
            \fill[black] (0,-3) circle (2pt);
            \node[right] at (0,-3) {$-1$};

            \draw[violet, dotted, thick] (-9,3) -- (0,3) -- (0,-3) -- (-9,-3) -- (-9,3);
            \draw[violet, thick] (-4,2) -- (0,2) -- (0,-2) -- (-4,-2) -- (-4,2);

            \draw[black, domain=-9:0, samples=100]
            plot (\x,{sqrt(-\x)})
            plot (\x,{-sqrt(-\x)});
            \node[left] at (-9,3) {$|x|=\sqrt{-t}$};

            \draw[black, domain=-9:0, samples=100]
            plot (\x,{0.4*sqrt(-\x)})
            plot (\x,{-0.4*sqrt(-\x)}););
            \node[left] at (-9,1.2) {$|x|=r\sqrt{-t}$};

            \fill[violet, fill opacity=.3]
            (-4,2) -- (0,2)
            -- plot[domain=0:-4, samples=100] (\x,{1.0*sqrt(-\x)})
            -- cycle;

            \fill[violet, fill opacity=.3]
            (-4,-2) -- (0,-2)
            -- plot[domain=0:-4, samples=100] (\x,{-1.0*sqrt(-\x)})
             -- cycle;

            \fill[gold, fill opacity=.2]
            plot[domain=-9:0, samples=100] (\x,{sqrt(-\x)})
            -- plot[domain=0:-9, samples=100] (\x,{0.4*sqrt(-\x)})
            -- cycle;

            \fill[gold, fill opacity=.2]
            plot[domain=-9:0, samples=100] (\x,{-sqrt(-\x)})
            -- plot[domain=0:-9, samples=100] (\x,{-0.4*sqrt(-\x)})
            -- cycle;
    \end{tikzpicture}
    \caption{Construction of $u$ in Theorem~\ref{thm: DGNM fails}. Within $Q_\eps(0,0)$ the solution is unbounded on the purple region. The singularity is driven by the coefficients that are complex-valued only inside the golden transition region.}
    \label{fig: parabolas}
\end{figure}
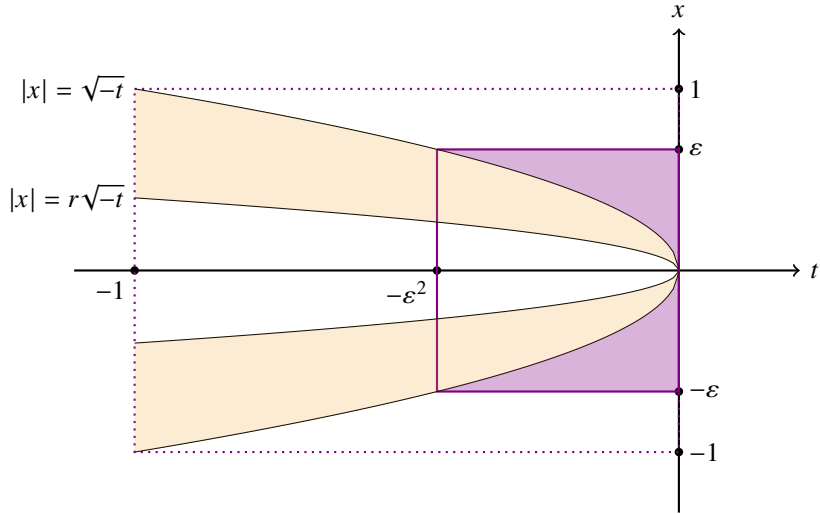
The equation for $F$ admits two elementary, \emph{real-valued} solutions: one satisfies the Neumann boundary condition at $y=0$, while the other one has the desired behavior as $y \to \infty$; see Section~\ref{sec: elementary solutions}. The centerpiece is a mechanism to connect the two solutions through an intermediate region. In Section~\ref{sec: gluing}, we recast this gluing problem as a control problem for a Riccati ODE in an auxiliary variable on a compact interval. The resulting endpoint compatibility conditions cannot be met over the reals. It is precisely at this stage --- and nowhere else in the construction --- that we pass to complex-valued $F$ and $A$.

\subsection*{Acknowledgment} At an early exploratory stage of this project, together with Lukas Niebel, we used OpenAI's GPT-5.6 Sol Ultra, produced what appeared to be a complicated counterexample with a computer-assisted proof, following Mooney's approach. The authors did not verify this output and did not use any arguments, ideas, or mathematical reasoning generated by the model in the subsequent mathematical analysis. The indication that a counterexample might exist, however, provided the initial motivation for the present project. 
\section{The profile equation}
\label{sec: profile}

We set $y \coloneqq \nicefrac{x}{\sqrt{-t}}$ and seek a self-similar solution to the parabolic equation \eqref{eq: PDE} of the form \eqref{eq: ansatz}. In this section, we derive sufficient properties on the profile $F$ that allow us to prove our main results from the introduction. With our ansatz, we have
\begin{align}
\label{eq: computations for self-similar solutions}
\begin{split}
	\partial_t u(t,x) &= (-t)^{-m-1} \Big(m F(y) + \frac{y}{2} F'(y) \Big) ,\\
	\partial_x(a \partial_x u) (t,x) 
    &= (-t)^{-m-1} (A F')'(y)
\end{split}
\end{align}
and we obtain a self-similar solution to \eqref{eq: PDE} provided that the right-hand sides above agree. We construct a particular pair $(F,A)$ with these properties in the following proposition. This is the main technical contribution of the paper, and we defer its proof to Sections \ref{sec: elementary solutions} and~\ref{sec: gluing}.
\begin{prop}
    \label{prop: profile equation}
    For every $m \in (0,\nicefrac{1}{4})$ there are functions $F, A: [0,\infty) \to \C \setminus \{0\}$ such that the following hold:
    \begin{enumerate}[(i)]
    	\item\label{item: i} $A$ is bounded, measurable and strictly accretive,
    	\item\label{item: ii} \label{item: growth bound F} $F$ is continuous and there is a constant $C$ depending on $m$ such that we have $\abs{F(y)} \leq C (1 + \abs{y})^{-2m}$ and $\abs{F'(y)} \leq C (1 + \abs{y})^{-2m-1}$ for a.e.\ $y \in [0,\infty)$,
    	\item\label{item: iii} $(AF')(0)=0$,
    	\item \label{item: Identity A, F} $F$ and $AF'$ are weakly differentiable and we have
    	\begin{align*}
    		(AF')'(y) = m F(y) + \frac{y}{2} F'(y), \qquad y \in [0,\infty).
    	\end{align*}
    \end{enumerate}
\end{prop}
\begin{rem}\label{rem: profile reflection}
 The even reflections of $F$ and $A$ yield a pair of functions which satisfy the properties in Proposition \ref{prop: profile equation} for every $y \in \R$.
\end{rem}
\begin{rem}
Note that we neither impose differentiability of $A$ nor of $F'$. 
\end{rem}
We need the following elementary observation on integrability of similarity profiles.
\begin{lem}
    \label{lem: integrability of self-similar solutions}
    Let $p \in [1,\infty)$, $\eps > 0$ and $\ell >0$ and let $G$ be a measurable function on $\R$ for which there exists a constant $C>0$ such that
    \begin{align*}
        \abs{G(y)} \leq C (1 + \abs{y})^{-2\ell}, \qquad y \in \R.
    \end{align*}
    Let $v(t,x) \coloneqq (-t)^{-\ell} G(\nicefrac{|x|}{\sqrt{-t}})$. If $p\ell < \nicefrac{3}{2}$, then $v \in \L^p(Q_\eps(0,0))$. The converse holds if $G$ is continuous at $0$ with $|G(0)| > 0$.
\end{lem}
\begin{proof}
    Substituting $y = \nicefrac{x}{\sqrt{-t}}$ gives
    \begin{align*}
        \norm{v}_{\L^p(Q_\eps(0,0))}^p \leq C \int_{-\eps^2}^0 (-t)^{-p\ell + \frac{1}{2}} \int_{-\nicefrac{\eps}{\sqrt{-t}}}^{\nicefrac{\eps}{\sqrt{-t}}} (1+ \abs{x})^{-2p\ell} \dd x \dd t.
    \end{align*}
    If $p \ell \leq \nicefrac{1}{2}$, then we bound the inner integral by $2 \nicefrac{\eps}{\sqrt{-t}}$ and the remaining integral in $t$ is finite. If $p \ell > \nicefrac{1}{2}$, then the inner integral is uniformly bounded with respect to $t$ and the remaining integral in $t$ is finite if in addition $p \ell < \nicefrac{3}{2}$. For the converse under the additional assumption on $G$, we find $\delta \in (0,\eps)$ and $c>0$ such that $\abs{G(y)}^p \geq c$ for all $\abs{y} < \delta$. We obtain
    \begin{align*}
        \infty> \norm{v}_{\L^p(Q_\eps(0,0))}^p \geq c \int_{-\eps^2}^0 \int_{-\delta\sqrt{-t}}^{\delta\sqrt{-t}} (-t)^{-p \ell} \dd x \dd t = 2 \delta c \int_{-\eps^2}^0 (-t)^{-p \ell + \frac{1}{2}} \dd t
    \end{align*}
    and therefore $p \ell < \nicefrac{3}{2}$. 
\end{proof}

\begin{proof}[Proof of Theorem~\ref{thm: DGNM fails}]
    We fix any $m\in(0,\nicefrac{1}{4})$, use the extended $F$, $A$ from Remark \ref{rem: profile reflection} and set
    \begin{align*}
    	a(t,x) \coloneqq A(\nicefrac{x}{\sqrt{-t}}),\qquad u(t,x) \coloneqq (-t)^{-m} F(\nicefrac{x}{\sqrt{-t}}).
    \end{align*}
    Then $a$ is bounded, measurable and strictly accretive. Moreover, we have
    \begin{align*}
        \partial_x u(t,x) = (-t)^{- m -1/2}  F'(\nicefrac{x}{\sqrt{-t}}).
    \end{align*} 
    Lemma~\ref{lem: integrability of self-similar solutions} yields $u, \partial_x u \in \L^2(Q_r(0,0))$ and $u \notin \L^\infty(Q_r(0,0))$ for every $r>0$. Finally, $u$ solves \eqref{eq: PDE} on $(-\infty, 0) \times \R$ in the sense of distributions, see \eqref{eq: computations for self-similar solutions}.
\end{proof}

\begin{proof}[Proof of Theorem~\ref{thm: Lp fails}]
    We define $a$ and $u$ as in the previous proof but use the specific parameter $m = \nicefrac{3}{(2p)}$ in Proposition~\ref{prop: profile equation}. Note that $p>6$ guarantees that $m< \nicefrac{1}{4}$. Lemma~\ref{lem: integrability of self-similar solutions} yields $\|u\|_{\L^p(Q_\eps(0,0))} = \infty$. 

    It remains to extend $u$ to a solution of a parabolic equation for positive times. The Lions--Magenes lemma (applied to a smooth truncation of $u$ in space) yields $u \in \Cont([-1,0]; \L_\loc^2(\R))$. Set $u_0 \coloneqq \lim_{t \nearrow 0} u(t, \cdot)$ and note that Proposition~\ref{prop: profile equation} yields $|u_0(x)| \leq C |x|^{-2m}$. In particular, $\sup_{x \in \R} \|u_0\|_{\L^2(x-1,x+1)} < \infty$, that is, $u_0$ is locally uniformly in $\L^2$. Hence, we can extend $u$ to positive times through the heat semigroup and set $a(t,x) \coloneqq 1$ for $t>0$, see, e.g., \cite[Section~2]{Arrieta}.
\end{proof}

\section{Proof of Proposition~\ref{prop: profile equation}: the profile at the endpoints}
\label{sec: elementary solutions}

In this section, we study the profile equation
\begin{align}
\label{eq: profile equation}
(AF')'(y) = m F(y) + \frac{y}{2} F'(y)
\end{align}
with $m \in (0,\nicefrac{1}{4})$ fixed and construct two solution pairs $(F,A)$ with $A$ bounded and strictly accretive: One will satisfy the Neumann boundary condition in Proposition~\ref{prop: profile equation}~\ref{item: iii} and the other one will have the asymptotics in \ref{prop: profile equation}~\ref{item: ii}. 

We introduce
\begin{align}
\label{eq: P}
 P \coloneqq AF'
\end{align}
and write \eqref{eq: profile equation} as 
\begin{align}
\label{eq: Cauchy-Euler}
    m F(y) + \frac{y}{2} F'(y) = P'(y),
\end{align}
which is a Cauchy--Euler ODE with integrating factor $y^{2m}$ and explicit solution formula
\begin{align*}
    F(y) = y^{-2m} \biggl(C+ 2 \int_1^y \sigma^{2m-1} P'(\sigma) \dd \sigma \biggr).
\end{align*}
This yields a simple algorithm to find solution pairs to \eqref{eq: profile equation}: Prescribe $P'$ such that the integral on the right can be computed explicitly, find a fully explicit solution $F$ to \eqref{eq: Cauchy-Euler} and, provided $F'$ stays away from $0$, reconstruct $A$ from \eqref{eq: P}.

Throughout this section, $r \in (0,1)$ and $c \in \C \setminus \{0\}$ are degrees of freedom that will be fixed at a later stage. Running the algorithm with the quadratic function
\begin{align*}
    P_0'(y) \coloneqq c \biggl(m + (m+1)\frac{y^2}{r^2} \biggr)
\end{align*}
yields the following pair $(F_0, A_0)$. (Alternatively, the reader can skip the algorithm and simply confirm that \eqref{eq: profile equation} holds.) 
\begin{lem}
    \label{lem: explicit solutions of profile zero}
    The following pair $(F,A) = (F_0, A_0)$ of functions on $[0,\infty)$ solves \eqref{eq: profile equation}:
    \begin{align*}
        F_0(y) \coloneqq c \biggl(1 + \frac{y^2}{r^2} \biggr), \qquad 
        A_0(y) \coloneqq \frac{mr^2}{2} + \frac{m+1}{6}y^2.
        \intertext{It corresponds to}
        F_0'(y) = \frac{2cy}{r^2}, \qquad 
        P_0(y) \coloneqq \frac{2cy}{r^2}\biggl(\frac{mr^2}{2} + \frac{m+1}{6}y^2 \biggr).
    \end{align*}
    In particular, $A_0$ is real-valued, continuous, bounded from below by $\nicefrac{mr^2}{2} > 0$, we have $(A_0F_0')(0) = 0$ and $F_0$ has no zeros.
\end{lem}
In order to obtain a solution $F_\infty$ with the right asymptotics, we initiate the algorithm with
\begin{align*}
    P_\infty'(y) \coloneqq m(2m+1)(1+y)^{-2m-2}
\end{align*}
and obtain:
\begin{lem}
    \label{lem: explicit solutions of profile infinity}
    The following pair $(F, A) = (F_\infty,A_\infty)$ of functions on $[0,\infty)$ solves \eqref{eq: profile equation}:
    \begin{align*}
        F_\infty(y) \coloneqq \frac{y+2m+1}{(1+y)^{2m+1}} \qquad 
        A_\infty(y) \coloneqq \frac{1+y}{2(y+2m+2)}.
        \intertext{It corresponds to}
        F_\infty'(y) = -2m \frac{y+2m+2}{(1+y)^{2m+2}}, \qquad 
        P_\infty(y) \coloneqq \frac{-m}{(1+y)^{2m+1}}.
    \end{align*}
    In particular, $A_\infty$ is real-valued, continuous and strictly increasing, $F_\infty$ has no zeros, and we have
    \begin{align*}
        \frac{1}{4m+4} \leq A_\infty(y) \leq \frac{1}{2}, \qquad |F_\infty(y)| \leq \frac{2}{(1+y)^{2m}}, \qquad |F_\infty'(y)| \leq \frac{2}{(1+y)^{2m+1}}.
    \end{align*}
\end{lem}
With the two explicit solution pairs at hand, we fix $A,F$ for the proof of Proposition~\ref{prop: profile equation} outside of the compact interval $[r,1]$ as follows.
\begin{defn}[Profile and coefficients at the endpoints]
    \label{defn: Profile and coefficients at the endpoints}
    Let $r \in (0,1)$, $c \in \C \setminus \{0\}$ and let $A_0, A_\infty, F_0, F_\infty$ be as above. Set
\begin{align*}
    F(y) \coloneqq \begin{cases}
        F_0(y) &\text{if }  0\leq y \leq r \\
        F_\infty(y) &\text{if } y \geq 1
    \end{cases}, \qquad 
    A(y) \coloneqq \begin{cases}
        A_0(y) &\text{if }  0\leq y \leq r \\
        A_\infty(y) &\text{if } y \geq 1
    \end{cases}.
\end{align*}
\end{defn}
\section{Proof of Proposition~\ref{prop: profile equation}: the profile away from the endpoints}
\label{sec: gluing}

In the previous section we have defined $A$ and $F$ on $[0,r] \cup [1,\infty)$ and verified the boundary conditions required in Proposition~\ref{prop: profile equation}. The goal of this section is to complete the construction by finding a solution pair to \eqref{eq: profile equation} on $[r,1]$ such that $F$ and $AF'$ agree with Definition~\ref{defn: Profile and coefficients at the endpoints} at $y=r$ and $y=1$.

It will be convenient to use the logarithmic variable $y= y(s) = r \e^s$, where $y \in [r,1]$ and, consequently, $s \in [0, \log(\nicefrac{1}{r})]$. Our central idea is to study a projective variable associated with $F$ and $P = AF'$ first, to wit,
\begin{align}
\label{eq: Z}
    Z(s) \coloneqq 1 - \frac{yF(y)}{2A(y) F'(y)}.
\end{align}
We also set
\begin{align*}
    U(s) \coloneqq \frac{y^2}{2 A(y)}
\end{align*}
and note that $A$ is bounded and strictly accretive on $[r,1]$ precisely if $U$ is bounded and strictly accretive on $[0, \log(\nicefrac{1}{r})]$. Differentiating $Z$ (with respect to $s$) yields
\begin{align}
\label{eq: Z primes}
\begin{split}
    Z'(s)
    &= - \frac{y F(y)}{2P(y)} - \frac{y^2 F'(y)}{2P(y)} + \frac{2y^2 F(y) P'(y)}{(2 P(y))^2} \\
    &= - \frac{y F(y)}{2 P(y)} - \frac{y^2}{2 A(y)} + \frac{y^2 F(y)}{(2 P(y))^2} \biggl(2 P'(y) - 2m F(y) - yF'(y) \biggr) \\&\quad + \frac{2m y^2 F(y)^2}{(2 P(y))^2} + \frac{y F(y)}{2 P(y)} \frac{y^2 F'(y)}{2 P(y)} \\
    &= Z(s)-1 - U(s) + 2m (Z(s)-1)^2 + U(s)(1-Z(s)) \\ &\quad +\frac{y^2 F(y)}{(2 P(y))^2} \biggl(2 P'(y) - 2m F(y) - yF'(y) \biggr)\\
    &= 2m (Z(s))^2 + \bigl(1-4m - U(s) \bigr) Z(s) + 2m -1 \\ &\quad +\frac{y^2 F(y)}{(2 P(y))^2} \biggl(2 (AF')'(y) - 2m F'(y) - yF'(y) \biggr),
\end{split}
\end{align}
where the term in brackets vanishes precisely if the pair $(F,A)$ solves \eqref{eq: profile equation}. If $Z$ avoids the value $1$ and $U$ is strictly accretive, then we have reconstruction formulas: We get $\log(F)$ from $(Z,U)$ by
\begin{align*}
\frac{\d}{\d s} \Big( \log(F(r \e^s)) \Big) = \frac{y F'(y)}{F(y)} = \frac{y^2 F'(y)}{y F(y)} = \frac{U(s)}{1-Z(s)},
\end{align*}
and $A$ by
\begin{align*}
     A(r \e^s) = \frac{r^2 \e^{2s}}{2 U(s)}.
\end{align*}
This yields
\begin{align*}
P (r \e^s)  = \frac{r \e^s F(r \e^s) }{2(1-Z(s))}.
\end{align*}
Since neither $F$ nor $P$ vanish, we see from \eqref{eq: Z primes} that $(F,A)$ reconstructed from $(Z,U)$ in that fashion solves \eqref{eq: profile equation} provided that $Z$ solves the Riccati ODE
\begin{align}
\label{eq: Riccati}
Z' = 2m Z^2 + \bigl(1-4m - U \bigr)Z + 2m -1
\end{align}
with coefficients depending on $U$.
The following lemma summarizes these observations.
\begin{lem}
    \label{lem: Riccati reduction}
Suppose that $r \in (0,1)$ and $U,Z: [0, \log(\nicefrac{1}{r})] \to \C$ have the following properties:
    \begin{enumerate}[(i)]
        \item\label{item: Riccati1} $U$ is bounded, measurable and strictly accretive,
        \item\label{item: Riccati2} $Z$ is continuous, piecewise continuously differentiable and avoids the values $0$ and $1$,
        \item\label{item: Riccati3} the pair $(U,Z)$ solves \eqref{eq: Riccati} on $[0,\log(\nicefrac{1}{r})]$.
    \end{enumerate}
Then for every $d \in \C \setminus \{0\}$ we get a solution pair $(F,A)$ to the profile equation \eqref{eq: profile equation} on $[r,1]$ by
\begin{align*}
    A(r \e^s) = \frac{r^2 \e^{2s}}{2U(s)}, \qquad F(r \e^s) = d \exp\biggl(\int_0^s \frac{U(t)} {1-Z(t)} \dd t \biggr).
\end{align*}
Moreover, $A$ is bounded, measurable and strictly accretive on $[r,1]$.
\end{lem}
The projective variable $Z$ from \eqref{eq: Z} also reveals why complex-valued functions are needed when trying to complete the profile from Definition~\ref{defn: Profile and coefficients at the endpoints}. Indeed, since we need $F$ and $P$ continuous at $y=r$ and $y=1$, the solution has to satisfy the compatibility conditions
\begin{align}
\label{eq: Z compatible}
\begin{split}
    Z(0) &= 1 -\frac{r F_0(r)}{2P_0(r)} = 1-\frac{3}{4m+1} \eqqcolon -\alpha,\\
    Z(\log(\nicefrac{1}{r})) &= 1- \frac{F_\infty(1)}{2P_\infty(1)} = 2 + \frac{1}{m} \eqqcolon \beta.
\end{split}
\end{align}
Note that $\alpha > 0$ since $m \in (0, \nicefrac{1}{4})$, and that $\beta>1$. However, a continuous, \emph{real-valued} $Z$ cannot connect $-\alpha$ to $\beta$ without attaining the value $1$ --- which is precisely what $Z$ is not allowed to do. It is at this point that we need to work over the complex numbers.

In order to construct a pair $(Z,U)$ as in Lemma~\ref{lem: Riccati reduction}, we solve \eqref{eq: Riccati} for $U$ and find
\begin{align}
\label{eq: U from Z}
U = 1 - 4m + 2mZ + \frac{2m-1}{Z} - \frac{Z'}{Z}.
\end{align}
Thus, we need to find $Z$ with property \ref{item: Riccati2} such that the right-hand side in \eqref{eq: U from Z} is strictly accretive. The right-hand side is also bounded and defining $U$ by \eqref{eq: U from Z} gives \ref{item: Riccati3} for free. We implement this strategy in the following lemma.
\begin{lem}
    \label{lem: Riccati solution}
    There exist $r \in (0,1)$ and functions $U,Z: [0, \log(\nicefrac{1}{r})] \to \C$ that satisfy properties \ref{item: Riccati1} -- \ref{item: Riccati3} of Lemma~\ref{lem: Riccati reduction} and, additionally,
    \begin{align*}
        Z(0) = -\alpha, \qquad Z(\log(\nicefrac{1}{r})) = \beta.
    \end{align*}
\end{lem}
\begin{proof}
As explained above, we only need to construct $Z$ with property \ref{item: Riccati2} such that $Z(0)=-\alpha$, $Z(\log(\nicefrac{1}{r})) = \beta$ and such that the right-hand side of \eqref{eq: U from Z} is strictly accretive.

For the construction we use polar coordinates $Z(s) = \varrho(s) \e^{\ii \theta(s)}$ with $\varrho(s) > 0$ and $\theta(s) \in [0,\pi]$. Taking real parts in \eqref{eq: U from Z} yields
\begin{align}
     \Re(U) &= 1 - 4m + \biggl(2m\varrho - \frac{1-2m}{\varrho} \biggr) \cos(\theta) - \frac{\varrho'}{\varrho}  \label{eq: Re of U}.
\end{align}
Note that for $\varrho > 0$ we have
\begin{align}
\label{eq: sign of cos factor}
    2m\varrho - \frac{1-2m}{\varrho} \leq 0 \quad \Longleftrightarrow \quad 2m \varrho^2 \leq 1 -2m 
    \Longleftrightarrow \quad \varrho \leq \sqrt{\frac{1}{2m}-1}
\end{align}
and equality holds precisely at the critical radius
\begin{align*}
\varrho_* \coloneqq\sqrt{\frac{1}{2m} -1}.
\end{align*}
We have $\varrho_* > 1$ since $m \in (0, \nicefrac{1}{4})$ and $\varrho_* < \beta$ since
$\varrho_*^2 < \nicefrac{1}{2m}$ and $\beta^2 > \nicefrac{4}{m}$. Moreover, $\alpha < \varrho_*$ since
\begin{align*}
    \varrho_*^2 - \alpha^2 = \frac{(1-2m)(32m^2+1)}{2m(4m+1)^2} >0.
\end{align*}
Hence, the order of $\alpha, \varrho_*, \beta$ on the real line is as illustrated in Figure~\ref{fig: Z}.
\begin{figure}[htb]
    \definecolor{gold}{HTML}{F1A226}
    \colorlet{color1}{violet}
    \colorlet{color2}{gold}
    \begin{tikzpicture}[scale=1, every node/.style={font=\small}]

            \draw[black, thick, ->] (-5,0) -- (8,0) node[right] {$\Re$};
            \draw[black, thick, ->] (0,-1) -- (0,4) node[above] {$\Im$};

            \fill[black] (-2,0) circle (2pt);
            \node[below] at (-2,0) {$-1$};
            \fill[black] (2,0) circle (2pt);
            \node[below] at (2,0) {$1$};
            \fill[black] (7,0) circle (2pt);
            \node[below] at (7,0) {$\beta$};
            \fill[black] (-1,0) circle (2pt);
            \node[below] at (-1,0) {$-\alpha$};
            \fill[black] (-3.5,0) circle (2pt);
            \node[below] at (-3.5,0) {$-\varrho_*$};
            \fill[black] (3.5,0) circle (2pt);
            \node[below] at (3.5,0) {$\varrho_*$};
            \fill[black] (0,3.5) circle (2pt);
            \node[above] at (0.3,3.4) {$\ii \varrho_*$};
            \node[above, violet] at (-3,2) {$Z$};

            \draw[very thick, violet, postaction={
                decorate,
                    decoration={
                        markings,
                        mark=at position 0.03 with {\arrow{>}},
                        mark=at position 0.2 with {\arrow{>}},
                        mark=at position 0.5 with {\arrow{>}},
                        mark=at position 0.8 with {\arrow{>}},
                        mark=at position 0.9 with {\arrow{>}}
            }}] (-1,0) -- (-3.5,0) arc[start angle=180, end angle=0, radius=3.5] -- (7,0);
            
    \end{tikzpicture}
    \caption{The orientation of $\alpha$, $\varrho_*$ and $\beta$ on the real line and the specific path $Z$  chosen in the proof of Lemma~\ref{lem: Riccati solution} to connect $-\alpha$ to $\beta$ in $\C \setminus \{0, 1\}$. The order of $\alpha$ and $1$ is immaterial for the argument and $\alpha < 1$ as displayed only holds when $m> \nicefrac{1}{8}$.}
    \label{fig: Z}
\end{figure}
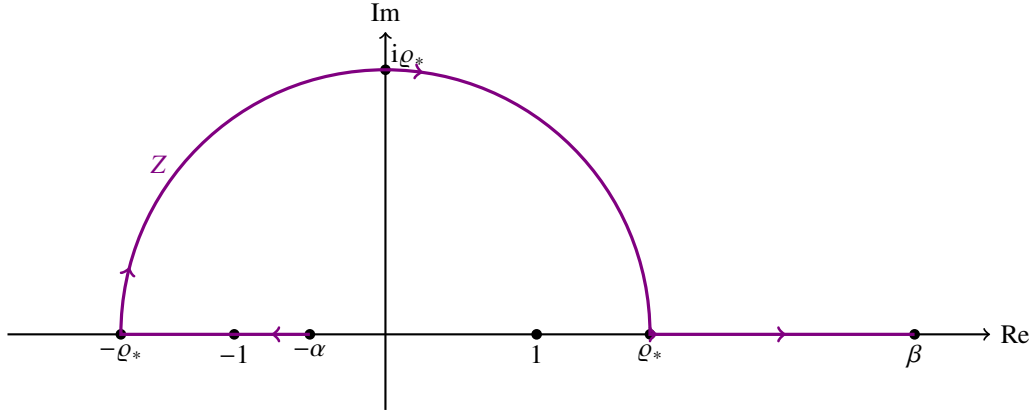
We now construct $Z$ by concatenating three continuously differentiable paths, as indicated in Figure~\ref{fig: Z}.

\emph{Step 1: $-\alpha \to - \varrho_*$.} We take $\theta=\pi$ constant and $\varrho$ to interpolate linearly from $\alpha$ to $\varrho_*$. In particular, $\varrho \leq \varrho_*$. Since we have the length $\log(\nicefrac{1}{r})$ of the parameter interval for $\varrho$ at our disposal, we can also assume $|\varrho'| \leq \alpha\frac{1-4m}{2}$ upon taking $r>0$ small. Combining \eqref{eq: Re of U} and \eqref{eq: sign of cos factor}, we find
\begin{align*}
     \Re(U) 
     = 1 - 4m + \frac{1-2m}{\varrho} - 2m\varrho - \frac{\varrho'}{\varrho}
     \geq 1 - 4m - \frac{|\varrho'|}{\alpha} 
     \geq \frac{1-4m}{2} > 0.
\end{align*}

\emph{Step 2: $- \varrho_* \to \varrho_*$.} We take $\varrho = \varrho_*$ constant and $\theta$ to interpolate linearly from $\pi$ to $0$. By the defining property of $\varrho_*$, we obtain from \eqref{eq: Re of U} that
\begin{align*}
     \Re(U)  \geq 1 - 4m > 0.
\end{align*}
This is the key feature of the critical radius: it allows us to move $Z$ in the upper half-plane without affecting strict accretivity of $U$.

\emph{Step 3: $\varrho_* \to \beta$.} We take $\theta=0$ constant and $\varrho$ to interpolate linearly from $\varrho_*$ to $\beta$. Thus, $\varrho \geq \varrho_*$ and, as in Step~1, we can assume $|\varrho'| \leq \frac{1-4m}{2}$ upon taking $r>0$ small. Combining \eqref{eq: Re of U} and \eqref{eq: sign of cos factor}, we find
\begin{align*}
     \Re(U) 
     = 1 - 4m + 2m\varrho - \frac{1-2m}{\varrho}- \frac{\varrho'}{\varrho}
     \geq 1 - 4m - |\varrho'|
     \geq \frac{1-4m}{2} > 0. &\qedhere
\end{align*}
\end{proof}
Finally, we give the
\begin{proof}[Proof of Proposition~\ref{prop: profile equation}]
    We fix $r$ as in Lemma~\ref{lem: Riccati solution}, let $(Z,U)$ be the corresponding pair, and complete the definition of $A$ and $F$ from Definition~\ref{defn: Profile and coefficients at the endpoints} by setting
    \begin{align*}
    F(y) \coloneqq \begin{cases}
        F_0(y) &\text{if }  0\leq y \leq r \\
        d \exp\biggl(\int_0^{\log{\nicefrac{y}{r}}} \frac{U(t)} {1-Z(t)} \dd t \biggr)& \text{if }  r < y < 1\\
        F_\infty(y) &\text{if } y \geq 1
    \end{cases}, \qquad 
    A(y) \coloneqq \begin{cases}
        A_0(y) &\text{if }  0\leq y \leq r \\
        \frac{y^2}{2U(\log(\nicefrac{y}{r}))}& \text{if }  r < y < 1\\
        A_\infty(y) &\text{if } y \geq 1
    \end{cases}.
    \end{align*}
    Next, we check that $F$ can be made continuous at the gluing points $y=r$ and $y=1$ by an appropriate choice of the remaining free parameters. From Definition~\ref{defn: Profile and coefficients at the endpoints}, we find
    \[
    \begin{aligned}
     F(r-) &= F_0(r) = 2c, &\qquad F(r+) &= d, \\[4pt]
     F(1-) &= d \exp\biggl(\int_0^{\log{\nicefrac{1}{r}}} \frac{U(t)}{1-Z(t)}\,\dd t \biggr),
     &\qquad F(1+) &= F_\infty(1) = \frac{2m+2}{2^{2m+1}}.
    \end{aligned}
    \]
    The second line determines $d \in \C \setminus \{0\}$ such that $F(1-) = F(1+)$ and with $c \coloneqq \nicefrac{d}{2}$ we get $F(r-) = F(r+)$. Once $F$ is continuous at the gluing points, continuity of $P$ follows from \eqref{eq: Z compatible} and we get property \ref{prop: profile equation}~\ref{item: Identity A, F}. Proposition~\ref{prop: profile equation}~\ref{item: i} is satisfied by construction and \ref{prop: profile equation}~\ref{item: iii} as well as \ref{prop: profile equation}~\ref{item: ii} outside of $[r,1]$ have already been taken care of in the previous section. On $[r,1]$, the estimates in \ref{prop: profile equation}~\ref{item: ii} reduce to $F$ and $F' = \nicefrac{P}{A}$ being bounded, which follows since $F$ and $P$ are continuous and $A$ is strictly accretive.
\end{proof}

\subsection*{Copyright}

A CC-BY 4.0 \url{https://creativecommons.org/licenses/by/4.0/} public copyright license has been applied by the authors to the present document and will be applied to all subsequent versions up to the Author Accepted Manuscript arising from this submission.

\bibliographystyle{abbrv}
\bibliography{references.bib}

\end{document}